\documentclass[12pt, a4paper]{amsproc}

\newcommand{\V}{{\mathfrak V}}

\newcommand{\Ni}{{\mathfrak N}}
\newcommand{\B}{{\mathfrak B}}
\newcommand{\M}{{\mathfrak M}}
\newcommand{\X}{{\mathfrak X}}
\newcommand{\Y}{{\mathfrak Y}}

\newcommand{\U}{{\mathfrak U}}
\newcommand{\A}{{\mathfrak A}}

\newcommand{\Z}{{\mathbb Z}}
\newcommand{\F}{{\mathbb F}}

\newcommand{\1}{\{1\}}

\newcommand{\var}[1]{\mathrm{var}\!\left( #1 \right)}
\newcommand{\varr}[1]{\mathrm{var}( #1 )}

\newcommand{\Wr}{\,\mathrm{Wr}\,}
\newcommand{\Wrr}{\,\mathrm{wr}\,}

 \newtheorem{thm}{Theorem}[section]
 
 \newtheorem{lem}[thm]{Lemma}
 
 \theoremstyle{definition}
 
 \theoremstyle{remark}
 \newtheorem{rem}[thm]{Remark}
 \newtheorem*{ex}{Example}
 \numberwithin{equation}{section}

\numberwithin{equation}{section}

\begin{document}

%%%%%%%%%%%%%%%%%%%%%%%%%%%%%%%%%%%%%%%%%%%%%%
%%%%%%%%%%%%%%%%%%%%%%%%%%%%%%%%%%%%%%%%%%%%%%
%%%%%%%%%%%%%%%%%%%%%%%%%%%%%%%%%%%%%%%%%%%%%%
$\phantom{Line}$ 


\title[A Classification Theorem for Varieties]{\large A Classification Theorem for Varieties Generated \\ by Wreath Products of Groups}

\author[V.H.~Mikaelian]{Vahagn H.~Mikaelian}
% \thanks{}

%\address{Informatics and Appl.~Mathhematics Department, Alex Manoogian 1, Yerevan State University, Yerevan 0025, Armenia.}

%\address{CSE, American University of Armenia (affiliate of the University of California L.A.), 40 Marshal Baghramyan Ave., Yerevan 0019, Armenia. \vskip3mm }

%%v \email{v.mikaelian@gmail.com}

\date{\today}

\dedicatory{\small To Professor Alexander Yu.~Ol'shanskii, my teacher, on his 70'th anniversary}

\begin{abstract}
We suggest a criterion classifying all the cases when for a nilpotent group $A$ of a restricted exponent and for any abelian group $B$ the variety $\var{A \Wr B}$ generated by the wreath product $A \Wr B$ is equal to the product variety $\var{A}\var{B}$. 
%
%Namely, the equality holds if and only if either the group $B$ is not of some non-zero exponent,
%
%or if $B$ is of a non-zero exponent $n$, and it contains a subgroup isomorphic to $C_{d}^c  \times C_{n/d}^\infty$, where  $c$ is the nilpotency class of $A$, \, $d$ is the largest divisor of $n$ coprime with $m$, \, $C_{d}^c$ is the direct power of $c$ copies of the cycle $C_d$ of order $d$, \, $C_{n/d}^\infty$ is the direct power of countably many copies of the cycle $C_{n/d}$ of order $n/d$. 
%
This continues our previous research on  varieties generated by wreath products of other classes of groups (abelian groups, finite groups, etc.). 
The obtained theorem generalizes some known results in the literature considering the same problem for more restricted cases. Some applications of the criterion also are discussed.
\end{abstract}

\maketitle

% {\small \tableofcontents }

%%%%%%%%%%%%%%%%%%%%%%%%%%%%%%%%%%%%%%%%%%%%%%
%%%%%%%%%%%%%%%%%%%%%%%%%%%%%%%%%%%%%%%%%%%%%%
%%%%%%%%%%%%%%%%%%%%%%%%%%%%%%%%%%%%%%%%%%%%%%
\section{Introduction}
\label{Introduction}

\noindent
The main aim of this work is to suggest a criterion under which for a nilpotent group of restricted exponent $A$ and for an abelian group $B$ the wreath product $A \Wr B$ generates the product of varieties $\var{A}$ and $\var{B}$ generated by $A$ and $B$ respectively, or in other notation, a condition under which the equality
\begin{equation}
\tag{$*$} 
\label{EQUATION_main}    
\var{A \Wr B} = \var{A} \var{B}
\end{equation}
holds for the given groups $A$ and $B$. By default all the wreath products here are assumed to be Cartesian (complete) wreath products, although the analogs of  statements are true for direct (restricted) wreath products also. 
Denote by $C_n$ the cyclic group of order $n$, and for any group $G$ denote by $G^k$ or by $G^\infty$ the direct product of $k$ or of countably many copies of $G$ respectively. In these notations our main result is:

\begin{thm}
\label{Theorem wr nilpotent abelian}%%%%%%%%%%%%%%%%%%%%%%%%%%%%%%%%%%%%%%%%%%%%%
Let $A$ be any nilpotent group of finite exponent $m$, and let $B$ be any abelian group. Then the equality \eqref{EQUATION_main} holds for $A$ and $B$ if and only if:
\begin{enumerate}
  \item[a)]  either the group $B$ is not of finite non-zero exponent;
  \item[b)] or $B$ is of some finite exponent $n > 0$, and it contains a subgroup isomorphic to the direct product $C_{d}^c  \times C_{n/d}^\infty$, where $c$ is the nilpotency class of $A$, and $d$ is the largest divisor of $n$ coprime with $m$.
\end{enumerate}
\end{thm}

If the abelian group $B$ is of finite exponent, then by Pr\"ufer's theorem~\cite{Robinson} it is a direct product of some finite cyclic subgroups. So the point (b) above just states that in that direct product the cycles of order $d$ and of order $n/d$ are present ``sufficiently many'' times: $B$ should contain at least $c$ copies of $C_{d}$ and infinitely many copies of $C_{n/d}$. In other words, if for a prime divisor $q$ of $n$ we denote by $q^v$ the highest degree of $q$ dividing $n$, then the group $B$ should contain $c$ copies of $C_{q^v}$, if $q$ is coprime with $m$, and infinitely many copies of $C_{q^v}$, if $q$ divides $m$ (see also Remark \ref{REMARK different roles of p}). 
%
% The suggested criterion is easy to apply, since it just requires to know the nilpotency class of $A$ and the direct decomposition of $B$ (check Section~\ref{examples applications} where we display how easy is to apply it to sample groups).

Theorem~\ref{Theorem wr nilpotent abelian} continues our research on classification of cases when \eqref{EQUATION_main} holds for groups $A$ and $B$ of certain classes of groups. In particular, in \cite{AwrB_paper, Metabelien} we gave a full classification for \eqref{EQUATION_main} holding for any {\it abelian} groups $A$ and $B$, and in \cite{wreath products algebra i logika, shmel'kin criterion} we classified the cases when $A$ and $B$ are any {\it finite} groups. 
% Development of this research was also reflected in theses \cite{Kand diso} and \cite{Doctor diso}. 
% Below in Section~\ref{SECTION comparision} we will restate two main theorems of that research to compare them with Theorem~\ref{Theorem wr nilpotent abelian}.

\vskip2mm

One of the oldest occurrence of the equality \eqref{EQUATION_main} is due to G.~Higman who considered \eqref{EQUATION_main} holding for 
$A=C_p$ and $B=C_n$ \cite{Some_remarks_on_varieties}. C.~Haughton covered the case of any cyclic $A,B$ (see ~\cite{HannaNeumann} and \cite{Burns65}).
In $\S$5 of \cite{3N} B.H.~Neumann,  H.~Neumann, P.M.~Neumann ask: 
\textit{``If the groups $A$, $B$ belong to the varieties $\U$, $\V$, respectively, then $A \Wr B$
(and hence also $A \Wrr B$) belongs to the product variety $\U \V$. If $A$ generates
$\U$ and $B$ generates $\V$, then one might hope that $A \Wr B$ generates $\U \V$; but
this is in general not the case''}.
Then they bring examples where \eqref{EQUATION_main} does or does not hold for some specific groups such as $p$-groups, free groups,  discriminating groups, infinite direct powers, etc. For these and some other earliest results see Hanna Neumann's monograph~\cite{HannaNeumann} and \cite{3N,B3, Burns65, B+3N,   Some_remarks_on_varieties}. Our criterion generalizes these and also some other  known results in literature in full or in part.

And in general, motivation of such study is explained by importance of wreath product as a key tool to study the product varieties of groups. 
For, the product $\U \V$ consists of extensions of all groups $A \in \U$ by all groups $B \in \V$, and if the equality \eqref{EQUATION_main} holds for some fixed groups $A$ and $B$ generating the varieties $\U$ and $\V$ respectively, then we can restrict ourselves to consideration of $\var{A \Wr B}$, which is easier to study rather than to explore all the extensions inside $\U \V$.
There are very many examples, when this approach is used (we listed some of them in \cite{Metabelien}).

\vskip2mm
Technique of the arguments below is based both on traditional theory of varieties of groups and on some newer approaches. In Section~\ref{k_p series} we suggest the idea of application of $K_p$-series and of D.~Shield's formula \cite{Shield nilpotent a, Shield nilpotent b} as tools to describe some non-nilpotent varieties via their finitely-generated nilpotent groups (see more details in  \cite{K_p-series}). Also, our usage of Fitting and Frattini subgroups is based not only on their well known properties reflected, for example, in~\cite{HannaNeumann, Robinson} but also on initial articles of W.~Gasch\"utz and H.~Fitting \cite{Gaschuetz Frattini, Gaschuetz aufloessbar, Fitting Gruppen endlicher Ordnung}: in their original work some results are formulated in sharper forms for certain specific cases, which are relevant in our study (compare constructions with critical groups and the application of Theorem of Clifford in Section~\ref{SECTION critical groups} below with constructions in \cite[Section 2 in Chapter 5]{HannaNeumann} or with \cite{Burns65}).

Below we without any definitions use the basic notions on varieties of groups, relatively free groups, verbal subgroups, discriminating groups, wreath products, etc. The background information can be found in \cite{HannaNeumann, Robinson}. 
% Some of the sections use further specific notions. We use them without definitions by just giving references to the literature.

\vskip2mm
{\bf Acknowledgments.} 
In initial steps of this work the case of wreath products of $p$-groups, which currently is covered by Section~\ref{k_p series}, was handled using the $K_p$-series. 
Later we discussed the topic with Prof. A.Yu.~Ol'shanskii, who  suggested a shorter proof for Lemma \ref{LEMMA condition p groups} by the arguments of \cite{Olshanskii Neumanns-Shmel'kin}. In Section \ref{k_p series} we present both proofs, and the most part of considerations with $K_p$-series is not included into the current text, we just provide references to \cite{K_p-series} where closer discussion on $K_p$-series is presented. 

%Partial results of this research were presented to the conferences {\it ``Mal'tsev Meeting'' International Conference}, Novosibirsk, Russia, November 10--13, 2014 and {\it ``Growth, Symbolic Dynamics and Combinatorics of Words in Groups''}, {\'{E}cole normale sup\'{e}rieure (ENS)}, Paris, France, 1--5 June, 2015.

Partial results of our work were presented as a main plenary talk to  the {\it ``Mal'tsev Meeting''} international conference, Novosibirsk, Russia, November 10--13, 2014.
The author was supported in part by joint grant 15RF-054 of RFBR and SCS MES RA, and by 15T-1A258 grant of SCS MES RA.

%%%%%%%%%%%%%%%%%%%%%%%%%%%%%%%%%%%%%%%%%%%%%%
%%%%%%%%%%%%%%%%%%%%%%%%%%%%%%%%%%%%%%%%%%%%%%
%%%%%%%%%%%%%%%%%%%%%%%%%%%%%%%%%%%%%%%%%%%%%%
\section{A characterization of the critical groups in $\var{A}\var{B}$}
\label{SECTION critical groups}

This section prepares the tools required to prove the sufficiency part of Theorem~\ref{Theorem wr nilpotent abelian}. We without any definitions use the notions of Frattini subgroup $\Phi (G)$ and the Fitting subgroup $F(G)$ of the given group $G$. Their definitions and basic properties can be found in W.~Gasch\"utz's important article  \cite{Gaschuetz Frattini}. 
For constructions of semisimple groups see H.~Fitting's work \cite{Fitting Gruppen endlicher Ordnung}. Concise summary of these concepts also is presented in \cite{Robinson}.
For definitions and basic properties of critical groups we refer to \cite{HannaNeumann, Burns65}.
 
Assume $A$ is a nilpotent group of class $c$ and of exponent $m \not= 1$, and $B$ is an abelian group of exponent $n \not= 1$. 
Take any critical group $K$ of the variety $\var{A}\var{B}$. It is an extension of a group $N \in \var{A}$ by some group $S \in \var{B}$, where $S$ as a finite abelian group has a direct decomposition $S = U \times U^*$, with $U$ being a subgroup generated by all the elements of exponents coprime to $m$. Our main technical goal is to establish some {\it bounds of the number of direct cyclic summands} of $U$ in terms of the nilpotency class $c$. Below we  without loss of generality may assume that both $S$ and $N$ are non-trivial, for, (a) if $S=\1$, then $K \in \var{A}$ and $U= \1$; \,  (b)  if $N=\1$, then $K \in \var{B}$ and according to \cite[51.36]{HannaNeumann} $K$ is a cyclic group of prime power order (as any critical abelian group), and so $U$ either is a non-trivial cyclic group or is trivial, so the bounds are evident.

\vskip2mm
We will twice use the result of W.~Gasch\"utz proved in \cite{Gaschuetz Frattini} as Satz 7: in any finite group $G$ any abelian normal subgroup $H$, intersecting with $\Phi (G)$ trivially, causes a factorization $G = H \tilde G$, where $\tilde G \not= G$, $H \cap \tilde G = \1$ and $H$ is a direct product of factors which are some minimal normal subgroups in $G$.
\vskip2mm

Denote $\Phi = \Phi (K)$ and $F = F (K)$ and notice that the Fitting subgroup $F$ is a $p$-group for some prime $p$ because (as a finite nilpotent group) $F$ is a direct product of its Sylow subgroups for distinct primes. These Sylow subgroups intersect trivially, and each of them is characteristic in $F$, and thus is normal in $K$. If $F$ contained more then one such Sylow subgroups, $K$ would be isomorphic to the subdirect product of its factor groups by those Sylow subgroups. Since $K$ is a critical group,  that option is ruled out.

Since $N$ is a normal nilpotent subgroup in $K$, it is contained in $F$. For the same reason $F$ also contains the normal (in fact also characteristic) subgroup $\Phi$. 
By \cite[Satz 2]{Gaschuetz Frattini} the factor group 
$$
\Phi(K/\Phi) = \Phi(K) / \Phi = \Phi / \Phi \cong \1
$$ 
is trivial (in terminology of W.~Gasch\"utz $K/\Phi$ is a {\it $\Phi$-free group}).
Our first application of Satz 7 is for the $\Phi$-free group $G = K/\Phi$ and $H = F/\Phi$. To apply it we yet have to show that $F/\Phi$ is abelian. By \cite[Satz 10]{Gaschuetz Frattini} the Fitting subgroup $F(K/\Phi)$ of $K/\Phi$ is equal to $F(K)/\Phi = F/\Phi$. Since $F/\Phi$ is normal in $K/\Phi$, by \cite[Satz 10]{Gaschuetz Frattini} the Frattini subgroup of $F/\Phi$ is in the Frattini subgroup of $K / \Phi$. So $F/\Phi$ also is a $\Phi$-free group. 
But a nilpotent finite group is $\Phi$-free only if it is a direct product of cycles of prime order (see \cite{Burnside powers of primes} or \cite[Abschnitt 3]{Gaschuetz Frattini}).
Thus by Satz 7 the factor-group $F/\Phi$ is a direct product of copies of a cycle $C_p$ for some prime $p$, and the factors of this direct product can be grouped so that each group is a minimal normal subgroup in the whole $K/\Phi$.

\vskip2mm
To approach to the second application of Satz 7 denote by $I$ the intersection $I = N \cap \Phi$, and notice that, since the subgroup $I$ is normal in $K$, we can apply Satz 2 from \cite{Gaschuetz Frattini} to get the Frattini subgroup of the factor-group $K/I$:
$$
\Phi(K/I)= \Phi(K)/I = \Phi/I.
$$
Clearly, the factor $N/I$ has a trivial intersection with $\Phi/I$. Also, $N/I$ is abelian, since 
$$
N/I = N/(N \cap \Phi)\cong N \Phi / \Phi \le F/ \Phi,
$$
where $F/ \Phi$ is an elementary abelian group. $N/I$ is normal in $K/I$, since $N$ is normal in $K$.
These preparations allows us to apply Satz 7 for $G = K/I$ and $H = N/I$. We get that $N/I$ is a direct product
\begin{equation}
\label{EQUATION_factorization of N/I}    
N/I = V_1/I\times \cdots \times V_s/I,
\end{equation}
where each $V_i/I$ is a minimal normal subgroup in $K/I$, and $V_i/I$ also is a direct product of some number of copies of the finite cycle $C_p$ (since $N/I$ is isomorphic to a subgroup of the elementary abelian group $F/\Phi$). Denote by $\tilde N /I$ the compliment of $N/I$ in $K/I$ (here $\tilde N$ is taken to be the full pre-image of that compliment under natural homomorphism with kernel $I$). We have: 
\begin{equation}
\label{EQUATION zerfaellen von K/I}
\text{$\tilde N /I \not= K/I$, \quad $N/I \, \cdot \, \tilde N /I = K/I$ \quad  and  \quad $N/I \cap \tilde N /I \cong \1$.}
\end{equation}
For now let us leave aside the case when $N/I$ is trivial, that is, when $N \le \Phi$ (we will consider this case a little later).

\vskip2mm
The constructions above allow us to apply a theorem of S.~Oates and M.B.~Powell, from \cite{Oates Powell} mentioned by Hanna Neumann in \cite{HannaNeumann} as Theorem 51.37 and stressed as a {\it ``much deeper result''}. Namely, if we  additionally assume $s > c$, then we will have:
\begin{enumerate}

\item[i)] $K = \langle \tilde N, V_1,\ldots, V_s  \rangle$ because $K$ is generated by $\tilde N \cup I \cup N$, and where $I$ can be omitted as it lies in the Frattini subgroup (the set of non-generators of $K$);

\item[ii)]  no proper subset of the set $\{V_1,\ldots, V_s\}$ suffices, together with $\tilde N$, to generate $K$ because each $V_i / I$ admits the actions by elements of $\tilde N$;

\item[iii)]  every mutual commutator group $[V_{\pi(1)},\ldots, V_{\pi(s)}]$, where $\pi$ is some permutation of the integers $1, \ldots , s$, is trivial because all factors $V_i$ are inside the nilpotent subgroup $N$ of class at most $c$. 
\end{enumerate}
Then by Theorem 51.37 the group $K$ is not critical. This contradiction implies that $s \le c$ (still with assumption $N \not\le \Phi$ made above).

Notice that we used Theorem 51.37 just for briefness of the argument, for, in current circumstances it would not be very complicated to directly show $K$ is not a critical group by application of methods with {\it special commutators} to the group $K$ (see Section 3 in Chapter 3 in \cite{HannaNeumann}, in particular, Lemma 33.35, Lemma 33.37 and Lemma 33.43).

\vskip4mm
The group $N/I$ and each factor $V_i/I$ of \eqref{EQUATION_factorization of N/I} can be considered to be finite-dimensio\-nal vector spaces over the field $\F_p$. Actions of elements of $\tilde N /I$ on $N/I$ by conjugations form a linear presentation $\psi$ of this group over the Galois field $\F_p$. In other words $N/I$ can be considered to be an $\F_p \tilde N /I$-module.

Since the direct factor $U$ of the group $S$ is of order coprime to $m$ and to $p$, then by Schur's Lemma $U$ can be identified by its copy in $K$. And since its order 
is coprime to $|I|$, we can identify that copy with its image in the factor group $\tilde N / I$. For simplicity let us use the same notation $U$ for this copy $UI/I$ also. 

Let us show that no non-trivial element $u \in U$ may centralize the space $N/I$, that is, the restriction of  $\psi$ to $U$ is a faithful presentation. Since $K/N \cong S$ is abelian, the factor-group $(K/I)/(N /I)$ also is abelian. By \eqref{EQUATION zerfaellen von K/I} that factor-group  is isomorphic to $\tilde N /I$. Thus if $u$ centralizes $N/I$, then it is easy to check that $\langle u \rangle $ is a normal subgroup in the whole group $K/I = N/I \, \cdot \, \tilde N /I$. But a normal and abelian subgroup $\langle u \rangle $ should lie in the Fitting subgroup $F(K/I)=F/I$, which is impossible since $\langle u \rangle $ contains elements of orders coprime to $p$.

Fix an index $i=1,\ldots, s$, and consider the linear presentation $\psi_i$ of $U$ in the subspace $V_i/I$ (the restriction of $\psi$ in $V_i/I$), that is, the $\F_p U$-module $V_i/I$.
Although $V_i/I$ is irreducible under actions (conjugations) of $\tilde N/I$ (or even of the whole group $K/I$), it may no longer be irreducible under actions of a smaller subgroup $U$. 
Let us show, that $V_i/I$ is a direct sum of isomorphic copies of a certain irreducible $\F_p U$-submodule. 

Since $\tilde N / I$ is abelian, $U$ is normal in it, and by Theorem of Clifford \cite[Theorem 8.1.3]{Robinson} $V_i/I$ can be presented as a direct sum of some irreducible $\F_p U$-submodules. Moreover, if $\{W_{i,1}, \ldots, W_{i,r_i}\}$ is a maximal subsystem of these submodules inducing all the pairwise non-isomorphic irreducible $\F_p U$-submodules of $V_i/I$, then one may group all the $\F_p U$-submodules in $V_i/I$ into blocks (called homogeneous components) $\overline W_{i,1}, \ldots, \overline W_{i,r_i}$ such that each $\overline W_{i,j}$, $j=1, \ldots, r_i$, is a sum of all irreducible $\F_p U$-submodules of $V_i/I$ isomorphic to $W_{i,j}$. By Theorem of Clifford the components $\overline W_{i,j}$ are conjugated modules, moreover, $\tilde N/I$ acts on them as a transitive permutation group. But since $U$ and $\tilde N/I$ are abelian, conjugation is identical operation here, and we have just one component: $r=1$ and
\begin{equation}
\label{EQUATION durect sum  V_i/I}
 V_i/I = \overline W_{i,1} = W_{i,1} \oplus \cdots \oplus W_{i,1}.
\end{equation}

Denote by $\rho_i$ the restriction of the presentation $\psi_i$ of $U$ in the subspace $W_{i,1}$. By our construction $\rho_i$ is irreducible. Denoting its kernel by $U_i$ we get that $U/U_i$ has an irreducible and faithful presentation in $W_{i,1}$. But an abelian group may have such a presentation only if it is cyclic. 
By Theorem of Remak:
\begin{equation}
\label{EQUATION decomposition of U/ V}
U/ V \cong U/U_1 \times \cdots \times U/U_s,
\end{equation}
where $V =\bigcap\{U_i \, | \, i=1,\ldots, s \}$. If an element $v$ is in $V$, it centralizes each of $W_{i,1}$. By decomposition \eqref{EQUATION durect sum  V_i/I} it also centralizes each of the direct summands $W_{i,1}$ in $\overline W_{i,1}$. Thus $v$ centralizes the sum of all $ V_i/I = \overline W_{i,1}$, $i=1,\ldots, r$, which is the whole space $N/I$. But, as we mentioned above, none of the non-trivial elements of $U$ may centralize $N/I$, and so $V \cong \1$. It follows from this and from \eqref{EQUATION decomposition of U/ V} that $U$ is a direct product of $s$ cyclic groups $U/U_1 , \ldots ,  U/U_s$. We earlier proved that $s \le c$, that is, $U$ is a direct product of at most $c$ cyclic groups, provided that the assumption $N \not\le \Phi$ holds. 

So it remains to cover the case when $N \le \Phi$. Since $K/N$ is abelian, $K/\Phi$ also is abelian. Then $K$ is nilpotent by \cite[Satz 10]{Gaschuetz Frattini}. As such, $K$ is a direct product of its Sylow subgroups by distinct primes. Since $K$ is a critical group, it consist of just one Sylow $p$-subgroup for the $p$ obtained from $N$ earlier. Thus $U \cong \1$, and the number of non-trivial direct cyclic factors is zero.

Assembling the bounds obtained in this section we get:

\begin{lem}
\label{Lemma Bound for critical}%%%%%%%%%%%%%%%%%%
Let $A$ be a nilpotent group of class $c$ and of exponent $m \not= 1$, and $B$ be an abelian group of exponent $n \not= 1$. 
Then any critical group $K$ of the variety $\var{A}\var{B}$ is an extension of a group from the variety  $\var{A}$ by some direct product $U \times U^*$ from the variety  $\var{B}$, such that
\begin{enumerate}
  \item the orders of non-trivial elements from $U$ are coprime to $m$;
  \item $U$ can be presented as a direct product of at most $c$ cyclic groups;
  \item the orders of non-trivial elements from $U^*$ can only be divided by primes, which are common divisors of $m$ and $n$.
\end{enumerate}
\end{lem}

%%%%%%%%%%%%%%%%%%%%%%%%%%%%%%%%%%%%%%%%%%%%%%
%%%%%%%%%%%%%%%%%%%%%%%%%%%%%%%%%%%%%%%%%%%%%%
%%%%%%%%%%%%%%%%%%%%%%%%%%%%%%%%%%%%%%%%%%%%%%
\section{The $K_p$-series and the case with wreath products of $p$-groups}
\label{k_p series}

In this section we consider the equality \eqref{EQUATION_main} for the case when $A$ is a nilpotent $p$-group of exponent $p^u$, and $B$ is an abelian group of exponent $p^v$, with $u,v >0$. This case is covered by Lemma~\ref{LEMMA condition p groups} below, and as we mentioned in the Introduction, we present two independent proofs for it. 

The notations and technique with $K_p$-series in this section are just for the first proof only, and they are not needed for understanding the rest of this paper. So we recommend not to focus on them and skip to Lemma~\ref{LEMMA condition p groups} unless the reader is interested to see our application of Shield's formula to varieties of groups (see details in \cite{K_p-series}).

G.~Baumslag showed that a wreath product of non-trivial groups $A$ and $B$ is nilpotent if and only if $A$ is a nilpotent $p$-group of finite exponent $p$, and $B$ is a finite $p$-groups~\cite{Baumslag nilp wr}. The exact  nilpotency class was computed by D.~Shield in \cite{Shield nilpotent a, Shield nilpotent b} (for background information and a survey see Chapter 4 of J.D.P.~Meldrum's monograph \cite{Meldrum book}). 
The $K_p$-series $K_{i,p}(B)$ of $B$ is defined for $i=1, 2, \ldots$ as
$ 
K_{i,p}(B) = 
\prod 
\{ \gamma_r(B)^{p^j} \mathrel{|}  \text{for all $r, j$ with $r p^{\,j} \ge i$} \}
$,
where $\gamma_r(B)$ is the $r$'th term of the lower central series of $B$. 
Let $d$ be the maximal integer, such that $K_{d,p}(B) \not= \{1\}$. Then for each $s=1,\ldots, d$ define $e(s)$ by
$
p^{e(s)} = |K_{s,p} / K_{s+1,p}|,
$
and also set:
$
a = 1 + (p-1) \sum_{s=1}^d \left(s \cdot e(s)\vphantom{a^b}\right)
$
and
$
b = (p-1)d.
$
By D.~Shield's theorem the nilpotency class of  $A \Wr B$ is the maximum
$   
\max_{h = 1, \ldots, \, c} \{
a \, h + (s(h)-1)b
\},
$
where $s(h)$ is chosen so that $p^{s(h)}$ is the exponent of the $h$'th term $\gamma_h(A)$ of the lower central series of $A$. 

\vskip2mm
Denote by $\beta$ the cardinality of $B$ and by $A^\beta$ the Cartesian product of $\beta$ copies of $A$. For the given fixed positive integer $l$ and for the integer $t \ge l$ denote 
$   
Z(l,t) =C_{p^v}^{l} \times C_{p^{v-1}}^{\, t-l}.
$
As we prove in \cite{K_p-series} there is a positive integer $t_0$, such that for all $t > t_0$ the nilpotency class of the wreath product
$A^\beta \Wr Z(l,t)$  is equal to: 
\begin{equation}
\label{EQUATION bound 1}  
\textstyle  
c + c\,t(p-1)\left( 
{1 - p^{v-1} \over \hskip-4mm 1-p} +l/t \cdot p^{v-1}
\right)
+ (\alpha-1)(p-1)p^{v-1},
\end{equation}
where the exponent of  $\gamma_c(A)$ is $p^\alpha$ ($\alpha \not= 0$, since the class of $A$ is c).

In $A$ we can construct a finite subgroup $\tilde A$ such that the exponents of terms $\gamma_h(\tilde A)$ and $\gamma_h(A)$ are equal for each $h = 1, \ldots, c$ %. This subgroup is easy to construct by choosing in each term $\gamma_h(A)$ an element $a_h$, such that $\exp{a_h}= \exp{\gamma_h(A)}$, and by taking some finitely many elements $a_{h,1}, \ldots, a_{h,r_h}\in A$, such that 
% $a_h \in \gamma_h \left(\langle a_{h,1}, \ldots, a_{h,r_h}\rangle \right)$. Then $\tilde A = \langle a_{h,i}, \ldots, a_{h,r_h} \,|\, h = 1, \ldots, c \, \rangle$ 
 \cite{K_p-series}.
If $\tilde A$ is a $z$-generator group, denote 
$
Y(z,t) =C_{p^v}^{\, t-z}
$.
We proved in \cite{K_p-series} that there is a positive integer $t_1$ such that for all $t > t_1$ the nilpotency class of the wreath product
$\tilde A \Wr Y(z,t)$   is equal to: 
\begin{equation}
\label{EQUATION bound 2}    
\textstyle
c + c(t-z)(p-1) \frac{1 - p^{v}}{\hskip-1mm 1-p}
+ (\alpha-1)(p-1)p^{v-1}.
\end{equation}

\begin{lem}
\label{LEMMA condition p groups}%%%%%%%%%%%%%%%%%%%%%%%%%%%%%%%%%%%%%%%%%%%%%
Let $A$ be a nilpotent $p$-group of exponent $p^u$, and $B$ be an abelian group of exponent $p^v$, with $u,v >0$. Then the wreath product $A \Wr B$ generates the variety $\var{A} \var{B} = \var{A} \A_{p^v}$ if and only if the group $B$ contains a subgroup isomorphic to the infinite direct power $C_{p^v}^\infty$ of the cycle  $C_{p^v}$.
\end{lem}

\begin{proof}[Proof based on $K_p$-series]
Sufficiency of the condition is easy to deduce from the discriminating property of $C_{p^v}^\infty$ (see \cite{B+3N} or Corollary 17.44 in \cite{HannaNeumann}).

To prove the necessity part, assume the group $B$ does not contain a subgroup isomorphic to $C_{p^v}^\infty$. By Pr\"ufer's theorem~\cite{Robinson} $B$ is a direct product of copies of some cyclic subgroups the orders of which are bounded by $p^v$. There are only finitely many, say $l$, such factors of order $p^v$, and collecting them together we get $B = B_1 \times B_2$, where  $B_1 = C_{p^v}^l$, and where $\exp {B_2}\le p^{v-1}$.

Any $t$-generator group $G \in \varr{A \Wr B}$ by \cite[16.31]{HannaNeumann} is in the variety generated by all the $t$-generator subgroups of $A \Wr B$.
Assume $T$ is one of such subgroups, and denote by $H$ its intersection with the base subgroup $A^B$ of $A \Wr B$. Then $T / H \cong (T A^B) / A^B \le (A \Wr B)/ A^B \cong B$, and so $T$ is an extension of $H$ by an at most $t$-generator subgroup $B'$ of $B= B_1 \times B_2$. By \cite{KaloujnineKrasner} $G$ is embeddable into $H \Wr B'$. 
$B'$ is a direct product of at most $t$ cycles, of which at most $l$ are of order $p^v$ and the rest are of  lower orders. So $B' \le Z(l,t)$ for a suitable $t$. Since $H \le A^\beta$, we get 
$
H \Wr B \in \var{A^\beta \Wr Z(l,t)}.
$
Thus the nilpotency class of $H \Wr B$ (and of $T$) are bounded by formula \eqref{EQUATION bound 1} for all $t > t_0$. 
Let us find a $t$-generator group in $\var{A} \var{B}$, which is of  higher class, at lest for some  $t$. The group  $\tilde A \Wr Y(z,t)$ is $t$-generator. For $t > t_1$ its nilpotency class is given by formula \eqref{EQUATION bound 2}. Notice that \eqref{EQUATION bound 1}  and \eqref{EQUATION bound 2} both consist of three summands of which the first and the third are the same. 
An easy comparison of the second summands (see \cite{K_p-series} for details) shows that the nilpotency class of the $t$-generator group  $\tilde A \Wr Y(z,t)$ from the variety $\var{A} \var{B}$ is higher than the maximum of the nilpotency classes of the $t$-generator groups in $\var{A \Wr B}$ for all  large enough $t$. Thus $\tilde A \Wr Y(z,t) \not\in \var{A \Wr B}$.
\end{proof}

Notice that the above proof {\it explicitly} constructs the group that belongs to the variety $\var{A} \var{B}$ but not to the variety $\var{A \Wr B}$ when  \eqref{EQUATION_main} fails to hold. 

\begin{proof}[Proof based on Ol'shanskii's theorem]
The sufficiency part of the proof for Lemma \ref{LEMMA condition p groups} can again be covered by the same remark about the discriminating property of $C_{p^v}^\infty$.

Again by Pr\"ufer's theorem~\cite{Robinson} the abelian group $B$ is a direct product of copies of some cyclic subgroups of $p$-power orders. If in this direct product the number of factors of order $p^v$ is finite, then the $p^{v-1}$-th power 
$$
C = B^{p^{v-1}}=\langle b^{p^{v-1}} \,|\, b\in B \rangle
$$ 
of $B$ is a finite group.

Thus the group $A \Wr B$ is an extension of a nilpotent $p$-group $A^B$ of finite exponent by means of the group $B$, which in turn is an extension of the finite $p$-group $C$ by the abelian $p$-group $B/C$ of exponent $p^{v-1}$. The full pre-image $\tilde C$ of $C$ under natural homomorphism $A \Wr B \to (A \Wr B)/A^B$ is a nilpotent $p$-group by G.~Baumslag's theorem~\cite{Baumslag nilp wr}.

Therefore $A \Wr B$ belongs to the product of the nilpotent variety $\Ni=\varr{\tilde C}$ and the abelian variety $\A_{p^{v-1}}$.
Assuming now that $\var{A \Wr B} = \var{A}\var{B}$ we have the inclusion 
\begin{equation}
\label{EQUATION inclusion}
\var{A}\var{B} \subset \Ni \, \A_{p^{v-1}}. 
\end{equation}
By A.Yu.~Olshanskii's theorem \cite{Olshanskii Neumanns-Shmel'kin}, if $\U \subseteq \V$ are varieties of groups with decompositions (into indecomposable factors) $\U=\U_1\cdots\U_k$ and $\V=\V_1\cdots\V_l$, then there exists an intermediate variety $\M$ (in the sense  $\U \subseteq \M \subseteq \V$) with two decompositions $\M=\M_1\cdots\M_k=\Ni_1\cdots\Ni_l$ such that $\U_i \subseteq \M_i$ for each $i$ and $\Ni_j\subseteq \V_j$ for each $j$.

By \eqref{EQUATION inclusion} we can apply this theorem for $\U=\U_1\U_2=\var{A}\var{B}$ and $\V=\V_1\V_2=\Ni \, \A_{p^{v-1}}$.
In particular, $\Ni_1 \subseteq \Ni$ and $\Ni_2 \subseteq \A_{p^{v-1}}$. Since nilpotent varieties are indecomposable, then $\M_1 = \Ni_1$ and $\M_2 = \Ni_2$. So $\var{B} \subseteq \M_2 \subseteq \A_{p^{v-1}}$. The later inclusion contradicts the assumption that $\exp{B} = p^v$.
\end{proof}

%%%%%%%%%%%%%%%%%%%%%%%%%%%%%%%%%%%%%%%%%%%%%%
%%%%%%%%%%%%%%%%%%%%%%%%%%%%%%%%%%%%%%%%%%%%%%
%%%%%%%%%%%%%%%%%%%%%%%%%%%%%%%%%%%%%%%%%%%%%%
\section{Sylow and Hall subgroups of groups in $\var{A \Wr B}$}
\label{Sylow and Hall}

This section contains some technical results and routine computations for Sylow and Hall subgroups in groups generated by wreath products. They will be used in the proof of Theorem~\ref{Theorem wr nilpotent abelian} in Section~\ref{The proof}.
Following the conventional notation we denote by  ${\sf Q}\X$, ${\sf S}\X$, ${\sf C}\X$ and  ${\sf D}\X$ the classes of all homomorphic images, subgroups, cartesian products and direct products of finitely many groups of $\X$ respectively. By Birkhoff's Theorem~\cite{HannaNeumann} for any class  of groups $\X$ the variety $\var{\X}$ generated by it can be obtained by these three operations: $\varr{\X}={\sf QSC}\,\X$.
For the given classes of groups $\X$ and $\Y$ denote $\X \Wr \Y = \{ X\Wr Y \,|\, X\in \X, Y\in \Y\}$. 

In two technical lemmas below we group a few statements which either are known facts in the literature, or are proved by us earlier (see Proposition 22.11 and Proposition 22.13 in \cite{HannaNeumann}, Lemma 1.1 and Lemma 1.2 in \cite{AwrB_paper} and also \cite{ShmelkinOnCrossVarieties}). Their proofs can be found in  \cite{AwrB_paper}. Both of these lemmas will be repeatedly used below.

\begin{lem}
\label{X*WrY_belongs_var}%%%%%%%%%%%%%%%%%%%%%%%%%%%%%%%%%%%%%%%%%%%%%
For arbitrary classes $\X$ and $\Y$ of groups and for arbitrary groups $X^*$ and $Y$, where either $X^*\in 
{\sf Q}\X$, or $X^*\in {\sf S}\X$, or $X^*\in {\sf C}\X$, and where $Y\in \Y$, the group $X^* \Wr
Y$  belongs to the variety $\var{\X \Wr \Y}$.
\end{lem}

\begin{lem}
\label{XWrY*_belongs_var}%%%%%%%%%%%%%%%%%%%%%%%%%%%%%%%%%%%%%%%%%%%%%
For arbitrary classes $\X$ and $\Y$ of groups and for arbitrary groups $X$ and $Y^*$, where $X\in\X$ and 
where $Y^*\in {\sf S}\Y$, the group $X \Wr Y^*$  belongs to the variety $\var{\X \Wr \Y}$.
Moreover, if  $\X$ is a class of abelian groups, then for each  $Y^*\in {\sf Q}\Y$ the group $X
\Wrr Y^*$  also belongs to $\var{\X \Wr \Y}$.
\end{lem}

\vskip3mm
Let $A$ be a nilpotent group of class $c$ and of exponent $m \not= 1$, and $B$ be an abelian group of exponent $n \not= 1$. Denote by $\U$ the variety generated by $A$, and assume $p_1, \ldots, p_s$ are all the prime divisors of $m$.
Since $\U \subseteq \Ni_{c,m} = \Ni_c \cap \B_m$ is a nilpotent variety, according to  \cite[Corollary 35.12]{HannaNeumann} it is generated by its $c$-generator free group $F = F_c(\U)$. Being a finite nilpotent group, $F$ is a direct product of its Sylow $p_i$-subgroups, $i=1,\ldots,s$:
\begin{equation}
\label{EQUATION_sylow_factors}    
F = S_{p_1} \times \cdots \times S_{p_s}.
\end{equation}
Notice that all the $p_i$'s actually participate: none of the Sylow subgroups is trivial, since $\exp(F) = \exp(\U) = \exp(A)$.

Assume there is a common prime divisor $p$ of $m$ and $n$ (it is one of the primes $p_i$), and let $p^v$ be the highest power of $p$ dividing $n$. 
Further, denote by $B(p)$ the $p$-primary component of the abelian group $B$ (the subgroup of all elements whose order is a power of the prime $p$). 
The following lemma allows to localize the Sylow $p$-subgroups of groups in the variety generated by $A \Wr B$:

\begin{lem}
\label{Lemma intersection}%%%%%%%%%%%%%%%%%%%%%%%%%%%%%%%%%%%%%%%%%%%%%
In the above notations the following equality holds:
\begin{equation}
\label{EQUATION_intersection1}    
\var{A \Wr B} \cap  \var{S_{p}} \A_{p^{v}}    = \var{S_{p} \! \Wr B(p)}.
\end{equation}
\end{lem}

\begin{proof}
The group $S_{p} \! \Wr B(p)$ evidently is in the product variety $\var{S_{p}} \A_{p^{v}}$. 
Since $S_{p}$ belongs to the variety $\var{F} = \var{A}$ and since $B(p)$ is a subgroup of $B$, then by Lemma~\ref{X*WrY_belongs_var} and Lemma~\ref{XWrY*_belongs_var} the group $S_{p} \! \Wr B(p)$ is in variety $\var{A \Wr B}$.
So the right-hand side of \eqref{EQUATION_intersection1} lies in the left-hand side.

To show the opposite inclusion notice that all groups of $\var{S_{p}} \A_{p^{v}}$ are $p$-groups, and it will be sufficient to take any $p$-group $P$ in $\var{A \Wr B}$, and to show that it is in $\var{S_{p} \Wr B(p)}$. Since $\var{A \Wr B}$ is locally finite, we may assume $P$ is finite. By  \cite[16.31]{HannaNeumann} $\var{A \Wr B}$ is generated by all the finitely generated subgroups $\{ R_i \, | \, i \in I\}$ of the group $A \Wr B$, which, clearly, all are finite also.

By G.~Higman's result \cite[Lemma 4.3, point (ii)]{Some_remarks_on_varieties} $P \in {\sf Q}{\sf S}{\sf D} \{ R_i \, | \, i \in I\}$, that is, there is a {\it finite} sequence of groups $R_1,\ldots, R_l$ (repetition of groups is allowed), such that $P$ is a surjective image of a subgroup $R$ of the direct product $R_{1} \times \cdots \times R_{l}$ of the groups $R_i$ under a homomorphism $\rho: R \to P$.
It is easy to show that $P$ is an image of a suitably chosen Sylow $p$-subgroup $\tilde P$ of $R$. This is true for even more general situation: for any finite group $G$ and for its homomorphic image $H=\varphi(G)$ any Sylow $p$-subgroup $P$ of $H$ is an image of a suitable Sylow $p$-subgroup $\tilde P$ of $G$. Indeed, the image $\varphi(Q)$ of any Sylow $p$-subgroup $Q$ of $G$ is a Sylow $p$-subgroup of $H$, and thus $h^{-1}  \varphi(Q) \, h = \varphi(Q)^h = P$ for some $h\in H$, so we can choose  $\tilde P = Q^h$. For the purposes of the current proof set $\tilde P$ be the respective Sylow $p$-subgroup of $R$ such that $\rho(\tilde P) = P$.

$\tilde P$ is the sub-direct product of its projections $\tilde P_i$ on direct factors $R_{i}$, $i = 1, \ldots , l$. Since each $R_{i}$ is some subgroup of $A \Wr B$, denote by $M_i$ the intersection of $\tilde P_i$ with the base group $A^B$. $M_i$ is normal in $\tilde P_i$, and one can consider the factor group $\tilde P_i / M_i$, which is isomorphic to some subgroup $N_i$ of $B$ because:
\begin{equation}
\label{EQUATION long line}    
\tilde P_i / M_i = \tilde P_i / (\tilde P_i \cap A^B) \cong (\tilde P_i \cdot A^B) / A^B \le A \Wr B / A^B \cong  B.
\end{equation}
By the Kaloujnine-Krassner theorem~\cite{KaloujnineKrasner} $\tilde P_i$ can be embedded into the wreath product $M_i \Wr N_i$. 

Being a $p$-subgroup in $B$, the group $N_i$ is a subgroup also in $B(p)$, and thus by Lemma~\ref{XWrY*_belongs_var} $\tilde P_i \in \var{M_i \Wr N_i} \subseteq \var{M_i \Wr B(p)} $ holds. Further, since $M_i \in \var{F}$, using G.~Higman's result again, we get $M_i \in {\sf Q}{\sf S}{\sf D} \{ F\}$, that is, $M_i$ is a surjective image of a subgroup $L$ of the direct product $F_{1} \times \cdots \times F_{k}$ of finitely many copies of the finite group $F$ under a homomorphism $\mu_i: L \to M_i$. The above used argument about the pre-image of a Sylow $p$-subgroup shows here that the $p$-group $M_i$ also is an image of a Sylow $p$-subgroup $\tilde M_i$ of $L$. The group $\tilde M_i$ is a sub-direct product of its projections $\tilde M_{i,r}$ on direct factors $F_{r}$, $r = 1, \ldots , k$. It follows from \eqref{EQUATION_sylow_factors} that each of $\tilde M_{i,r}$ is embeddable into $S_{p}$. Thus $M_i \in \var{S_{p}}$ and $M_i \Wr B(p) \in \var{S_{p} \Wr B(p)} $ by Lemma~\ref{X*WrY_belongs_var}.
\end{proof}

If $n$ is a divisor of $m$, then Lemma \ref{EQUATION_intersection1}   together with Lemma~\ref{LEMMA condition p groups} already allows to classify all cases, when \eqref{EQUATION_main} holds for a nilpotent group $A$ of finite exponent $m$ and for an abelian group of exponent $n$. 

\vskip3mm
The open case (when $n$ has prime divisors, not dividing $m$) can be covered by the following lemma, the proof of which is sketched, since it is similar to the previous proof.
Let $p$ be a prime divisor of $m$ not dividing $n$, let $q$ be a prime divisor of $n$ not dividing $m$, and assume $q^v$ is the highest power of $q$ dividing $n$. Let $B(q)$ be the $q$-primary component of $B$, and let $F$ and $S_{p_1}, \ldots , S_{p_s}$ denote the same as in Lemma~\ref{Lemma intersection}. 
Then the following lemma localizes the Hall $\{p,q\}$-subgroups in groups of $\var{A \Wr B}$:

\begin{lem}
\label{Lemma intersection2}%%%%%%%%%%%%%%%%%%%%%%%%%%%%%%%%%%%%%%%%%%%%%
In the above notations the following equality holds:
\begin{equation}
\label{EQUATION_intersection2}    
\var{A \Wr B} \cap  \var{S_{p}} \A_{q^{v}}    = \var{S_{p} \! \Wr B(q)}.
\end{equation}
\end{lem}

\begin{proof}[Proof sketch]
By arguments similar to the first paragraph of the previous proof the right side of \eqref{EQUATION_intersection2} lies in the left side. 

To show the opposite inclusion denote  $\pi = \{ p, q\}$ and notice that it will be sufficient to take any $\pi$-group $P$ in $\var{A \Wr B}$, and to show it is in $\var{S_{p} \Wr B(q)}$.  
Again, it is enough to consider finite groups $P$ only, and to assume by  \cite[16.31]{HannaNeumann} that $P$ is a surjective image of a subgroup $R$ of the direct product $R_{1} \times \cdots \times R_{l}$ under a homomorphism $\rho: R \to P$ for some finite subgroups $R_i$ of $A \Wr B$.

For the next step we need some modifications, since the Hall $\pi$-subgroups not always satisfy the properties we used above for Sylow $p$-subgroups.
The groups we deal with are at most nilpotent-by-abelian, and thus they are soluble. For the finite soluble groups the theorems of P.~Hall \cite{Robinson} provide the tools we need: for any finite soluble group $G$ and for its homomorphic image $H=\varphi(G)$ any Hall $\pi$-subgroup $P$ of $H$ is an image of a Hall $\pi$-subgroup $\tilde P$. For, the image $\varphi(Q)$ of any Hall $\pi$-subgroup $Q$ of $G$ is a Hall $\pi$-subgroup of $H$, and thus $h^{-1}  \varphi(Q) \,h = \varphi(Q)^h = P$ for some $h\in H$, and we can choose  $\tilde P = Q^h$. 
For the current proof set $\tilde P$ be the respective Hall $\pi$-subgroup of $R$, such that $\rho(\tilde P) = P$.

$\tilde P$ is the sub-direct product of its projections $\tilde P_i$ on direct factors $R_{i}$, $i = 1, \ldots , l$. Like in previous proof denote $M_i = \tilde P_i \cap A^B$. Calculations similar to \eqref{EQUATION long line} show that the factor group $\tilde P_i / M_i$ is isomorphic to some subgroup $N_i$ of $B$. Thus the $\pi$-group $\tilde P_i$ can be embedded into the wreath product $M_i \Wr N_i$. 

Being an image of a $\pi$-group, $N_i$ also is a $\pi$-group. Since $p$ does not divide $n$, the group $N_i$ is inside $B(q)$, and thus by Lemma~\ref{XWrY*_belongs_var} $\tilde P_i \in \var{M_i \Wr N_i} \subseteq \var{M_i \Wr B(q)} $ holds.

Since $M_i \in \var{F}$, using G.~Higman's result for one more time we get that $M_i$ is a surjective image of a subgroup $L \le F_{1} \times \cdots \times F_{k}$ (of copies of $F$) under a homomorphism $\mu: L \to M_i$. The above used argument about the pre-image of a Hall $\pi$-subgroup shows that the $\pi$-group $M_i$ also is an image of a Hall $\pi$-subgroup $\tilde M_i$ of $L$. Since $q$ does not divide $m$, the group $\tilde M_i$ actually is a $p$-group. 
$\tilde M_i$ is a sub-direct product of its projections $\tilde M_{i,r}$ on direct factors $F_{r}$, $r = 1, \ldots , k$, and each of $\tilde M_{i,r}$ is embeddable into $S_{p}$. Thus $M_i \in \var{S_{p}}$ and $M_i \Wr B(q) \in \var{S_{p} \Wr B(q)} $ by Lemma~\ref{X*WrY_belongs_var}.
\end{proof}

Notice that these two lemmas do not cover the case when the exponent of $A$ has a prime divisor $p$, which does not divide the exponent of $B$. In this connection see Remark~\ref{REMARK different roles of p} on ``three roles'' of prime divisors of $m$ and $n$ below.

%%%%%%%%%%%%%%%%%%%%%%%%%%%%%%%%%%%%%%%%%%%%%%
%%%%%%%%%%%%%%%%%%%%%%%%%%%%%%%%%%%%%%%%%%%%%%
%%%%%%%%%%%%%%%%%%%%%%%%%%%%%%%%%%%%%%%%%%%%%%
\section{The proof of Theorem~\ref{Theorem wr nilpotent abelian}}
\label{The proof}

The facts collected in previous sections allow to assemble the proof for our main result:

\begin{proof}[Proof of Theorem~\ref{Theorem wr nilpotent abelian}]
Let us start with statement (a) of Theorem~\ref{Theorem wr nilpotent abelian}, when the group $B$ is not of finite exponent. If $B$ contains an element $z$ of infinite exponent, then it contains the infinite cyclic group $\langle z \rangle = C \cong \Z$, which is a discriminating group \cite[17.6]{HannaNeumann}. Since $\var{\langle z \rangle} =\var {B} = \A$, then by  \cite[17.5]{HannaNeumann} the group $B$ also discriminates $\A$. It remains to apply \cite[22.42]{HannaNeumann} to get that $A \Wr B$ generates $\var{A}\var{B}$.

Now assume all the elements of $B$ are of finite non-zero exponents, which in this case have no common upper bound. $B$ generates $\A$, and it will be enough to show that $B$ again is a discriminating group.
Since $C$ discriminates $\A$, any finite system of non-identities $w_1, \ldots, w_d$ of $\A$ can simultaneously be falsified on some  set $z_1, \ldots , z_s \in C$. Take an element $z \in B$ with exponent larger than the values of all words $w_i$ on variables $z_1, \ldots , z_s$. The non-identities will simultaneously be falsified on some elements in the subgroup $\langle z \rangle \le B$.

\vskip3mm

Turning to the proof of necessity for the condition at statement (b) of Theorem~\ref{Theorem wr nilpotent abelian},  notice that it will be enough to consider two cases: either if $B$ fails to contain $C_{d}^c$, or if it fails to contain $C_{n/d}^\infty$.

\vskip1mm
{\it Case 1}. If $B$ does not contain a subgroup isomorphic to $C_{d}^c$, there is at least one prime divisor $q$ of $n$, such that $B$ does not contain $C_{q^{v}}^c$, where $q^{v}$ is the exponent of $B(q)$, that is, the highest power of $q$ dividing $n$. Using the notation of \eqref{EQUATION_sylow_factors}, the exponent $m$ has at least one prime divisor $p = p_i$, such that the nilpotency class of the Sylow $p$-subgroup $S_p = S_{p_i}$ is equal to $c$ (otherwise the class of $A$ would be lower than $c$).

We have $\var{A \Wr B} \cap \, \var{S_{p}} \A_{q^{v}}   =  \var{S_{p} \! \Wr B(q)}$ by Lemma~\ref{Lemma intersection2}. 
By Lemma~\ref{Lemma Bound for critical} any critical group of $\var{S_{p}} \A_{q^{v}}$ is an extension of some group from $\var{S_{p}}$ by a direct product of at most $c$ cyclic groups from $\A_{q^{v}}$. 
Recalling the steps of the proof of Lemma~\ref{Lemma intersection2} we see that this direct product can be taken to be a subgroup of $B(q)$. 
Thus in \eqref{EQUATION_intersection2} we can replace $B(q)$ by some \textit{finite} (at most $c$-generator) subgroup $B^*$ of $B(q)$. 

According to Theorem 1 in \cite{shmel'kin criterion} (which we also bring in Section~\ref{SECTION comparision} below as Theorem~\ref{Theorem wr finite}) the variety $\var{S_{p} \! \Wr B(q)}=\var{S_{p} \! \Wr B^*}$ is strictly less than the product $\var{S_{p}} \var{B(q)} = \var{S_{p}} \A_{q^{v}}$ because $B$, $ B(q)$ and $B^*$ do not possess a subgroup, isomorphic to $C_{q^v}^c$, where $c$ is the nilpotency class of $S_{p}$ (that is, the class of $A$). Therefore the equality \eqref{EQUATION_main} does not hold, since the product $\var{A} \var{B}$ does contain $\var{S_{p}} \var{B(q)}$ as a subvariety.

\vskip1mm
{\it Case 2}. If $B$ does not contain a subgroup isomorphic to $C_{n/d}^\infty$, there is at least one {\it common} prime divisor $p$ of $m$ and of $n$, such that $B$ does not posses $C_{p^{v}}^\infty$, where $p^{v}$ is the exponent of $B(p)$ and the highest power of $p$ dividing $n$. 
In this case the nilpotency class of the Sylow $p$-subgroup $S_p = S_{p_i}$ may be less than $c$, but $S_p$ certainly is non-trivial, otherwise the exponent $m$ would not be divisible by $p$.

This time we use the equality $\var{A \Wr B} \cap  \var{S_{p}} \A_{p^{v}}    = \var{S_{p} \! \Wr B(p)}$ of  Lemma~\ref{Lemma intersection}. According to Lemma~\ref{LEMMA condition p groups} in Section~\ref{k_p series} the variety $\var{S_{p} \! \Wr B(p)}$ is strictly less than the product $\var{S_{p}} \var{B(p)} = \var{S_{p}} \A_{p^{v}}$ because $B$ and $ B(p)$ do not possess a subgroup isomorphic to $C_{p}^\infty$. Therefore like in Case 1 above, the equa\-lity \eqref{EQUATION_main} does not hold, since $\var{A} \var{B}$ does contain $\var{S_{p}} \var{B(p)}$.

\vskip3mm

It remains to prove sufficiency of the condition at statement (b) in Theorem~\ref{Theorem wr nilpotent abelian}. Assume $B$ contains a subgroup isomorphic to $C_{d}^c \times C_{n/d}^\infty$, and take an arbitrary critical group $K$ in $\var{A}\var{B}$. By Lemma~\ref{Lemma Bound for critical} $K$ is an extension of some group $N \in \var{A}$ by some group $S \in \var{B}$, and $S = U \times U^*$, where $U$ and $U^*$ meet the points (1)--(3) of Lemma~\ref{Lemma Bound for critical}. In particular, by points (1) and (2) $U$ is a subgroup of $C_{d}^c$, and by point (3) $U^*$ is a subgroup of $C_{n/d}^\infty$. So we have
$$
	S = U \times U^* \le C_{d}^c \times C_{n/d}^\infty \le B,
$$
and therefore by Lemma~\ref{X*WrY_belongs_var} and Lemma~\ref{XWrY*_belongs_var} the group $K$ belongs to the variety $\var{A \Wr B}$.

The product $\var{A}\var{B}$ is a locally finite variety, and it is generated by its critical groups \cite[51.41]{HannaNeumann}. Since they all belong to the variety $\var{A \Wr B}$, the equality \eqref{EQUATION_main} does hold for the considered case.

The proof of Theorem~\ref{Theorem wr nilpotent abelian} is completed.
\end{proof}

\begin{rem}
\label{REMARK different roles of p}%%%%%%%%%%%%%%%%%%%%%%%%%%%%%%%%%%%%%%%%%%%%%
Notice an interesting pattern about ``three roles'' of the prime divisors of the exponents $m$ and $n$ according to Theorem~\ref{Theorem wr nilpotent abelian}: 

a) The most important role for feasibility of equality \eqref{EQUATION_main} belongs to the primes $p$ dividing both $m$ and $n$: for each of them the group $B$ should contain the infinite direct power $C_{p^v}^\infty$ of the cycle $C_{p^v}$, where $p^v$ is the highest power of $p$ dividing $n$.

b) Less demanding are the primes $q$ that divide $n$ but not $m$: for each of them the group $B$ should contain the direct power $C_{q^v}^c$, with $c$ being the class of $A$, and $q^v$ being the highest power of $q$ dividing $n$.

c) And, finally, the primes $p$ that divide $m$ but not $n$ have no impact on feasibility of \eqref{EQUATION_main} (they just participate in determination of the exponent of $\var{A}$).
\end{rem}

%%%%%%%%%%%%%%%%%%%%%%%%%%%%%%%%%%%%%%%%%%%%%%
%%%%%%%%%%%%%%%%%%%%%%%%%%%%%%%%%%%%%%%%%%%%%%
%%%%%%%%%%%%%%%%%%%%%%%%%%%%%%%%%%%%%%%%%%%%%%
\section{Some comparison with earlier results}
\label{SECTION comparision}

Let us restate two main theorems of earlier research  \cite{AwrB_paper}--\cite{shmel'kin criterion} and show how Theorem~\ref{Theorem wr nilpotent abelian} generalizes them for wide classes of groups. Firstly, restrict our consideration to {\it finite} groups (recall that we denoted by $C_n$ the cycle of order $n$):

\begin{thm}[Theorem 1 in \cite{shmel'kin criterion}]
\label{Theorem wr finite}%%%%%%%%%%%%%%%%%%%%%%%%%%%%%%%%%%%%%%%%%%%%%
For finite non-trivial groups $A$ and $B$ the equality~\eqref{EQUATION_main} holds if and only if: 
\begin{enumerate}
  \item[a)]  \vskip-1mm the exponents of group $A$ and $B$ are coprime;
  \item[b)] $A$ is a nilpotent group, $B$ is an abelian group;  
  \item[c)] $B$ contains a subgroup isomorphic to the direct power $C_n^c$, where $c$ is the nilpotency class of $A$, and $n$ is the exponent of $B$.
\end{enumerate}
\end{thm}

That the condition (b) of Theorem \ref{Theorem wr finite} is necessary, follows from A.L.~Shmel'kin's theorem~\cite[Theorem 6.3]{ShmelkinOnCrossVarieties}. 
But if $A$ is nilpotent (and of finite exponent, since it is finite) and $B$ is abelian, then the point (b) of Theorem~\ref{Theorem wr nilpotent abelian} requires that $B$ contains the direct product $C_{d}^c  \times C_{n/d}^\infty$, where $d$ is the largest divisor of $n$, coprime with $m$. But since $B$ is finite, it contains no nontrivial subgroup $C_{n/d}^\infty$. So we get $n/d = 1$, that is, the condition (a) of Theorem \ref{Theorem wr finite}.

\vskip2mm

The second direction for restriction is to consider Theorem~\ref{Theorem wr nilpotent abelian} for {\it abelian} groups. We had proved:

\begin{thm}[Theorem 6.1 in~\cite{AwrB_paper} or Theorem A in~\cite{Metabelien}]
\label{classification for abelian groups}
For abelian non-trivial groups $A$ and $B$ the equality~\eqref{EQUATION_main} holds if and only if:
\begin{enumerate}
  \item[a)] either at least one of the groups $A$ and $B$ is not of finite  exponent;
  \item[b)] or if  $\exp A = m$  and $\exp B = n$ are both finite, and for every  common prime divisor $p$ of $m$ and $n$ a direct decomposition of the $p$-primary component $B(p)$ of $B$ contains infinitely many direct summands $C_{p^v}$, where $p^v$ is the highest power of $p$ dividing $n$.
\end{enumerate}
\end{thm}

To make this theorem more intact with Theorem~\ref{Theorem wr nilpotent abelian} and Theorem~\ref{Theorem wr finite}, let us re-formulate it slightly differently.  
For every  common prime divisor $p$ of $m$ and $n$ the component $B(p)$ contains infinitely many direct summands $C_{p^v}$ if and only if $B(p)$ contains the direct product of all of them, that is, if $B$ contains the direct power $C_{n/d}^\infty$, with $d$ defined in Theorem~\ref{Theorem wr nilpotent abelian}. Thus:

\begin{thm}[an equivalent form for Theorem~\ref{classification for abelian groups}]
\label{Equivalent form}
For abelian non-trivial groups $A$ and $B$ the equality~\eqref{EQUATION_main} holds if and only if:
\begin{enumerate}
  \item[a)] either at least one of the groups $A$ and $B$ is not of finite  exponent;
  \item[b)] or if  $\exp A = m$  and $\exp B = n$ are both finite, and $B$ contains a subgroup isomorphic to the infinite direct power $C_{n/d}^\infty$, where $d$ is the largest divisor of $n$ coprime with $m$.
\end{enumerate}
\end{thm}

Since here $A$ is nilpotent, we have $c = 1$, and the point (b) in Theorem~\ref{Theorem wr nilpotent abelian}  would require that $B$ contains the product $C_{d}  \times C_{n/d}^\infty$ (just one factor $C_{d}$). So this may seem to be a stronger requirement than the point (b) in Theorem~\ref{Equivalent form}. However, by elementary properties of abelian groups, for any prime divisor $p$ of $n$ the group $B$ contains at least one subgroup isomorphic to $C_{p^v}$. The direct product of some of these $C_{p^v}$ (for all $p$ coprime to $m$) does provide the direct summand $C_{d}$. So for abelian groups Theorem~\ref{Theorem wr nilpotent abelian} is a partial generalization of this case also.

%%%%%%%%%%%%%%%%%%%%%%%%%%%%%%%%%%%%%%%%%%%%%%%%%%%%%%%%%%%%%%%%%%%%%%
%%%%%%%%%%%%%%%%%%%%%%%%%%%%%%%%%%%%%%%%%%%%%%%%%%%%%%%%%%%%%%%%%%%%%%
%%%%%%%%%%%%%%%%%%%%%%%%%%%%%%%%%%%%%%%%%%%%%%%%%%%%%%%%%%%%%%%%%%%%%%
\section{Examples and applications of the criterion}  
\label{examples applications}%%%%%%%%%%%%%%%%%%%%%%%%%%%%%%%%%%%%%%%%%%%%%%%%%%%%%%%%%%%

Some number of the examples, in which \eqref{EQUATION_main} holds or does not hold for wreath products of abelian groups or of finite groups, can be obtained by Theorem 6.1 in \cite{AwrB_paper}, Theorem A in \cite{Metabelien} or Theorem 1 in \cite{shmel'kin criterion}. 
Here we just give references to them: examples 4.6, 5.4, 6.3, 6.4 in \cite{AwrB_paper}, 
examples 6.9, 7.5, 8.5, 8.6 in \cite{Metabelien},
examples 1, 2 in \cite{shmel'kin criterion} (repetitions allowed). 
The listed examples can be handled and generalized by Theorem \ref{Theorem wr nilpotent abelian}. In particular:

\begin{ex}
\label{EXAMPLE dihedral}
In \cite{Kovacs dihedral} L.G.~Kov{\'a}cs  has computed the variety generated by dihedral group $D_4$ of order $8$:
% ($D_4$ is a nilpotent group of class $2$)
$\var {D_4}= \A_2^2 \cap \Ni_2$. 
In \cite{shmel'kin criterion} we have seen that for any {\it odd} $n$ the equality \eqref{EQUATION_main} holds for $A = D_4$ and for finite abelian group $B$ of exponent $n$ if and only if $B$ contains $C_n^2$. And, if \eqref{EQUATION_main} holds, then $\var{A \Wr B} = (\A_2^2 \cap \Ni_2)\A_n$.

Now we can generalize this in two directions. First, if $n$ is {\it even}, present it in the form $n = d \cdot 2^h$, where $d$ is odd (we can assume $d\not=0$ since that is the case of odd $n$ covered in \cite{shmel'kin criterion}). Then by Theorem~\ref{Theorem wr nilpotent abelian} \eqref{EQUATION_main} holds for $A = D_4$ and for any abelian group $B$ of finite exponent if and only if $B$ contains the direct product $C_{d}^2 \times C_{2^h}^\infty$ (which, clearly, is isomorphic to $C_{n}^2 \times C_{2^h}^\infty$). So, for instance,   \eqref{EQUATION_main} does not hold if $B = C_6 = C_3 \times C_2$,\, if $B = C_3^\infty \times C_2$ or if $B = C_3^k \times C_2^2$ (for any $k \ge 1$),\, but \eqref{EQUATION_main} does  hold if $B = C_3^\infty \times C_2^7$.

As a second direction for generalization we can replace $A = D_4$  by its finite or infinite direct power $D_4^k$ or $D_4^\infty$. Then the requirement for $B$ remains unchanged.
\end{ex}

\begin{ex}
\label{EXAMPLE quaternion}
The quaternion group $Q_8$ of order $8$ also is nilpotent of class $2$, and it generates the same variety $\var{Q_8} = \var{D_4} = \A_2^2 \cap \Ni_2$ (see~\cite{HannaNeumann}).  Thus examples similar to the points of Example~\ref{EXAMPLE quaternion} can be constructed for the group $A=Q_8$.
\end{ex}

\begin{ex}
\label{EXAMPLE nilpotent}
We above denoted $\Ni_{c,m} = \Ni_c \cap \B_m$. Using simple properties of critical groups from~\cite{Burns65} and Proposition 2 from~\cite{wreath products algebra i logika} we saw in~\cite{wreath products algebra i logika} that 
\eqref{EQUATION_main} does not hold when $A=F_2(\Ni_{2,3})$ and $B = C_2$, and \eqref{EQUATION_main} holds if we take  $B = C_2^k$ (for any $k \ge 2$) instead.
Using the technique of R.G.~Burns with critical groups in~\cite[Section 3]{Burns65} one could build analogs of this example for any variety $\Ni_{2,p}$ and $C_q$, where prime numbers $p$ and $q$ are chosen so that $q$ divides $p-1$.

But much more general cases were covered by Theorem 1 in \cite{shmel'kin criterion}. Namely let $m, n >1$ be any coprime integers, and let $s \ge c$. 
Then for $A=F_s(\Ni_{c,m})$ and for an abelian group $B$ of exponent $n$ the wreath product $A \Wr B$ generates $\Ni_{c,m}\A_n$ if and only if  $B$ contains the direct product $C_n^r$ (we required $s \ge c$ to ensure $\var A=\Ni_{c,m}$).

Again, we can generalize this example in two directions. Firstly, if $n$ is not coprime to $m$, present it as $n=d \cdot n/d$, where $d$ is the largest divisor of $n$, coprime to $m$ (clearly $n/d\not=1$).
In these circumstances \eqref{EQUATION_main} holds with an abelian group $B$ of  exponent $n$, and $A \Wr B$ generates $\Ni_{c,m}\A_n$, if and only if $B$ contains $C_{d}^c \times C_{n/d}^\infty$. 

Secondly, we can consider relatively free groups of $\Ni_{c,m}$ of infinite rank, which were not covered in \cite{shmel'kin criterion}.
By Theorem~\ref{Theorem wr nilpotent abelian} the equality \eqref{EQUATION_main} holds for $A=F_\infty(\Ni_{c,m})$ and for the group $B$ mentioned above if and only if $B$ contains $C_{d}^c \times C_{n/d}^\infty$. 
\end{ex}

Since there is no shortage in examples of nilpotent groups which are either finite or have finite exponents, one may continue this list of examples, as the criterion of Theorem~\ref{Theorem wr nilpotent abelian} is easy-to-use: it just evolves the nilpotency class of $A$ and the direct decomposition of the abelilan group $B$.

\vskip3mm
Theorem~\ref{Theorem wr nilpotent abelian} also allows to add something our old research topic on $\circ$-products $\V \circ B$ and $A \circ \U$ (see the Ph.D. thesis~\cite{Kand diso}, advisor A.Yu.~Ol'shanskii, M.S.U.).
For the given variety $\V$ and the group $B$ the variety $\V \circ B$ is that generated by extensions of all groups from $\V$ by the group $B$, and $A \circ \U$ is the variety, generated by all extensions of the group $A$ by all groups from $\U$. In \cite{Kand diso} we used these $\circ$-products as tools to study the product varieties via consideration of cases, when $\V \circ B = \V \, \var{B}$ or $A \circ \U = \var{A}\U$. Now we have one more case:

\begin{thm}
Let $\V$ be any non-trivial nilpotent variety of finite exponent and let $B$ be any abelian group. Then the equality $\V \circ B = \V \, \var{B}$ holds if and only if  either the group $B$ is not of finite non-zero exponent,\,\,
or if $B$ is of some finite exponent $n > 0$, and it contains a subgroup isomorphic to the direct product $C_{d}^c  \times C_{n/d}^\infty$, where $c$ is the nilpotency class of $\V$, and $d$ is the largest divisor of $n$ coprime with $m$.
\end{thm}

%%% http://groupprops.subwiki.org/wiki/Group_of_nilpotency_class_two

\vskip4mm

\noindent
e-mail: v.mikaelian@gmail.com 


\begin{thebibliography}{88}




\bibitem{Baumslag nilp wr}
G.~Baumslag,
\textit{Wreath products and $p$-groups},
Proc. Camb. Philos. Soc. 55 (1959), 224--231.

\bibitem{B3}
G.~Baumslag,
\textit{Wreath products and extensions},
Math. Z., 81 (1963), 286--299.


\bibitem{B+3N}
G.~Baumslag, B.H.~Neumann,  Hanna Neumann, P.M.~Neumann
\textit{On varieties generated by finitely generated group},
Math. Z., 86 (1964), 93--122.

%\bibitem{BirkhoffQSC}
%G.~Birkhoff,
%\textit{On the structure of abstract algebras},
%Proc. Cambridge Phil. Soc., 31 (1935), 433--454.

%\bibitem{BrumbergOnWreathProducts}
%N.R.~Brumberg,
%\textit{Connection of wreath product with other operations on groups},
%Sib. Mat. Zh., 4 (1963), 6, 1221--1234 (Russian).

\bibitem{Burns65}
 R.G.~Burns,
{\it Verbal wreath products and certain product varieties of groups},
J. Austral. Math. Soc.  7 (1967), 356--374.


\bibitem{Burnside powers of primes}
W.~Burnside,
{\it On some properties of groups whose orders are powers of primes},
Lond. Math. Soc. Proc.(2) 13 (1913), 6--12.


\bibitem{Fitting Gruppen endlicher Ordnung}
H.~Fitting,
{\it Beitraege zur Theorie der Gruppen endlicher Ordnung} (German),
Jahresber. Dtsch. Math.-Ver. 48 (1938), 77--141.



\bibitem{Gaschuetz Frattini}
W.~Gasch\"utz,
{\it \"Uber die $\Phi$-Untergruppe endlicher Gruppen} (German),
Math. Z. 58 (1953), 160--170.

\bibitem{Gaschuetz aufloessbar}
W. Gasch\"utz,
{\it Zur Theorie der endlichen aufl\"osbaren Gruppen} (German),
Math. Z. 80 (1963), 300--305.



%\bibitem{BurnsDiso}
%R.G.~Burns,
%{\it Some Applications of Wreath Products of Groups} 
%Ph.D. thesis, Australian National University, 1966. 



%\bibitem{CohenMetabelian}
%D.E.~Cohen, 
%{\it On the laws of a metabelian variety}, 
%J. Algebra  5  1967 267--273.


%\bibitem{Gorenstein}
%D.~Gorenstein, 
%{\it Finite Groups}, second edition, 
%Chelsea Publishing Co., New York, 1980.

 

\bibitem{Some_remarks_on_varieties}
G.~Higman,
\textit{Some remarks on varieties of groups},
Quart. J. Math. Oxford, (2) 10 (1959), 165--178.



%\bibitem{HoughtonDiso}
%C.H.~Houghton,
%{\it Varieties of Groups and Wreath Products} 
%Ph.D. thesis, University of Manchester, UMIST, 1969. 




\bibitem{KaloujnineKrasner}
L.~Kaloujnine, M.~Krasner,
\textit{Produit complete des groupes de permutations et le probl\`eme d'extension des groupes, III},
Acta Sci. Math. Szeged, 14 (1951), 69--82.

%\bibitem{Kargapolov Merzlyakov}
%M.I. Kargapolov, Yu.I. Merzlyakov,
%\textit{Fundamentals of group theory} 4th ed. (Russian), Moscow, Nauka, Fizmatlit 1996.
%\hskip3mm
%English translation from the 2nd Russian ed. by Robert G.~Burns, Graduate Texts in Mathematics 62, New York-Heidelberg-Berlin: Springer-Verlag, XVII 1979.


\bibitem{Kovacs dihedral}
L.G.~Kov{\'a}cs,
\textit{Free groups in a dihedral variety},
Proc. Roy. Irish Acad. Sect. A  89  (1989), no. 1, 115--117.


%\bibitem{Liebeck_Nilpotency_classes}
%H.~Liebeck,
%\textit{Concerning nilpotent wreath products},
%Proc. Cambridge Phil. Soc., 58 (1962), 443--451.


%\bibitem{Marconi nilpotent}
%R.~Marconi,
%\textit{On the nilpotency class of wreath products} (Italian),
%Boll. Unione Mat. Ital., VI. Ser., D, Algebra Geom. 2 (1983), 1, 9--20.

%\bibitem{Meldrum nilpotent}
%J.D.P.~Meldrum,
%\textit{On nilpotent wreath products},
%Proc. Cambridge Philos. Soc., 68 (1970), 1--15.

\bibitem{Meldrum book}
J.D.P.~Meldrum,
\textit{Wreath products of groups and semigroups},
Pitman Monographs and Surveys in Pure and Applied Mathematics, 74, Harlow, Essex: Longman Group Ltd. xii 1995.





%%%%%%%%%%%%%%%%%%%%%%%%%%%%%%%%%%%%%%%%%%%%%%

\bibitem{Kand diso}
V.H.~Mikaelian,
\textit{The Identities of Finite Extensions of Groups}, 
Ph.D. (Cand. Sci.) thesis, M.V. Lomonosov Moscow State University, 1994.

% \bibitem{Doctor diso}
% V.H.~Mikaelian,
% \textit{Verbal Embeddings and Wreath Products of Groups}, 
% Doctor Sci. thesis, M.V. Lomonosov Moscow State University, 2011.

\bibitem{SubnormalEmbeddingTheorems}
V.H.~Mikaelian,
\textit{ Subnormal embedding theorems for groups,}
J. London Math. Soc., 62 (2000), 398--406. 

\bibitem{AwrB_paper}
V.H.~Mikaelian,
\textit{On varieties of groups generated by wreath products of abelian groups}, 
Abelian groups, rings and modules (Perth, Australia, 2000), Contemp. Math., 273, Amer. Math. Soc., Providence, RI (2001), 223--238.

%\bibitem{finitely generated abelian}
%V.H.~Mikaelian,
%\textit{On wreath products of finitely generated abelian groups}, 
%Advances in Group Theory, Proc. Internat. Research Bimester dedicated to the memory of Reinhold Baer, (Napoli, Italy, May-June, 2002), Aracne, Roma, 2003, 13--24.

%\bibitem{Two problems}
%V.H.~Mikaelian,
%\textit{Two problems on varieties of groups generated by wreath products of groups},  
%Int. J. Math. Sci.  31  (2002), no. 2, 65--75. 

\bibitem{Metabelien}
V.H.~Mikaelian,
\textit{Metabelian varieties of groups and wreath products of abelian groups}, 
Journal of Algebra, 2007 (313), 2, 455--485. 

%\bibitem{wreath products Prilozh}
%V.H.~Mikaelian,
%\textit{Varieties generated by wreath products of abelian groups}, 
%translated from Sovrem. Mat. Prilozh., Vol. 83, 2012. J. Math. Sci. (N. Y.)  195  (2013), no. 4, 523--528.

\bibitem{wreath products algebra i logika}
V.H.~Mikaelian,
\textit{Varieties generated by wreath products of abelian and nilpotent groups} (Russian), 
Algebra i Logika, 54 (2015), 1, 103--108. English translation appeared in Algebra and Logic, 54 (2015), 1, 70--73.

\bibitem{shmel'kin criterion}
V.H.~Mikaelian,
\textit{The Criterion of Shmel'kin and Varieties Generated by Wreath Products of Finite Groups}, 
accepted for publication in Algebra i Logika, to appear in 2016, see ArXiv:1503.08474.

\bibitem{K_p-series}
V.H.~Mikaelian,
\textit{On $K_p$-series and varieties generated by wreath products of $p$-groups}, 
in preparation, see ArXiv:1505.06293.








\bibitem{abelian subgr in metab gr}
V.H. Mikaelian, A.Yu. Olshanskii
{\it  On abelian subgroups of finitely generated metabelian groups}, 
Journal of Group Theory, {\bf 16} (2013), 695--705.

%%%%%%%%%%%%%%%%%%%%%%%%%%%%%%%%%%%%%%%%%%%%%%


\bibitem{HannaNeumann}
Hanna Neumann, 
{\it  Varieties of Groups}, 
Varieties of groups (Ergebn. Math. Grenzg., 37), Berlin-Heidelberg-New York, Springer-Verlag 1967.

\bibitem{3N}
B.H. Neumann, Hanna Neumann, Peter M.Neumann, \textit{Wreath products and varieties of groups},  
Math. Zeitschrift 80 (1962), 44--62

\bibitem{Oates Powell}
S.~Oates, M.B.~Powell,
\textit{Identical relations in finite groups}, 
J. Algebra 1 (1964), 11-39.

%\bibitem{Olshanskii basis}
%A.Yu.~Olshanskii
%\textit{On the problem of finite basis of identities in groups}, 
%Izv. Akad. Nauk SSSR, Ser. Math. 34 (1970), no. 2, 376--384; English translation in Math. USSR Izv.


% \bibitem{Olshanskii all finite groups are abelian}
% A.Yu.~Olshanskii
% \textit{Varieties in which all finite groups are abelian (Russian)}, 
% Mat. Sb. (N.S.)  126 (168)  (1985), no. 1, 59--82, 143.

%\bibitem{Ol'shanskii Kaluzhnin - Krasner}
%A.Yu. Olshanskii, 
%\textit{On Kaluzhnin-Krasner's embedding of groups,} 
%Algebra Discrete Math. 19 (2015), No. 1, 77--86.

\bibitem{Olshanskii Neumanns-Shmel'kin}
A.Yu.~Olshanskii,
\textit{The Neumanns-Shmel'kin's theorem (Russian)}, 
Vestnik Mosk. Univ., Ser. Matem.(1986), N 6, 61--64.


\bibitem{Robinson}
D.J.S.~Robinson, 
\textit{A Course in the Theory of Groups}, second edition, 
Springer-Verlag, New York, Berlin, Heidelberg 1996.


\bibitem{Shield nilpotent a}
D.~Shield,
\textit{Power and commutator structure of groups}, 
Bull. Austral. Math. Soc., 17 (1977), 1--52.

\bibitem{Shield nilpotent b}
D.~Shield,
\textit{The class of a nilpotent wreath product}, 
Bull. Austral. Math. Soc., 17 (1977), 53--89.

\bibitem{ShmelkinOnCrossVarieties}
A.L.~Shmel'kin,
\textit{Wreath products and varieties of groups},
Izv. AN SSSR, ser. matem., 29 (1965), 149--170 (Russian).
Summary in English: Soviet Mathematics. Vol. 5. No. 4 (1964). 

\end{thebibliography}
\end{document}